\documentclass
{amsart}
\usepackage{graphicx}
\newtheorem{thm}{Theorem}[section]
\theoremstyle{definition}
\newtheorem{cor}[thm]{Corollary}
\newtheorem{prop}[thm]{Proposition}
\newtheorem{defn}[thm]{Definition}
\newtheorem{lem}[thm]{Lemma}

\newtheorem{rem}[thm]{Remark}
\newtheorem{ex}[thm]{Example}

\numberwithin{equation}{section}
\begin{document}
\title[On the second spectrum of a module (II)]
{On the second spectrum of a module (II)}
\author{H. Ansari-Toroghy, S. Keyvani and F. Farshadifar}
\address{Department of pure Mathematics, Faculty of mathematical Sciences, University of Guilan,
P. O. Box 41335-19141 Rasht, Iran.} %
\email{ansari@guilan.ac.ir}%
\address{Department of pure Mathematics, Faculty of mathematical Sciences, University of Guilan,
P. O. Box 41335-19141 Rasht, Iran.} %
\email{siamak\_Keyvani@guilan.ac.ir}%
\address{Department of pure Mathematics, Faculty of mathematical Sciences, University of Guilan,
P. O. Box 41335-19141 Rasht, Iran.} %
\email{farshadifar@guilan.ac.ir}%

\subjclass[2000]{13C13, 13C99}%
\keywords {Second submodule, maximal second submodule, second spectrum, spectral space, and Zariski topology}

\begin{abstract}
Let $R$ be a commutative ring and $M$ an $R$-module. Let $Spec^s(M)$
be the the collection of all second submodules of $M$. In this article, we consider a
new topology on $Spec^s(M)$, called the second classical Zariski topology, and investigate
the interplay between the module theoretic properties of $M$ and the topological properties of $Spec^s(M)$.
Moreover, we study $Spec^s(M)$ from point of view of spectral space.
\end{abstract}

\maketitle
\section{Introduction}
Throughout this paper, $R$ will denote a commutative ring with
identity. If $N$ is a subset of an $R$-module $M$ we write $N\leq M$ to indicate that $N$ is a submodule of $M$.

Let $M$ be an $R$-module. A proper submodule $N$ of $M$ is said to be \emph{prime} if for
any $r \in R$ and $m \in M$ with $rm \in N$, we have $m \in N$
or $r \in (N:_RM)$. The \emph{prime spectrum} of $M$ denoted by $Spec(M)$ is the set of all prime submodules of $M$.

A non-zero submodule $N$ of $M$ is said to be \emph{second} if for each $a \in R$, the homomorphism $ N \stackrel {a} \rightarrow N$ is either surjective
or zero \cite{Y01}. More information about this class of modules can be found in \cite{AF12}, \cite{AF1}, and \cite{AF21}.

The concept of prime submodule has led to the development of topologies on the spectrum of modules. A brief history of this development can be seen in \cite[Page 808]{lu10}. More information concerning the spectrum of rings, posets, and modules can be found in \cite{Ab10}, \cite{Ab11}, \cite{BH08}, \cite{BHa0}, \cite{Ho71}, \cite{lS03}, and \cite{S10}.

Let $Spec^s(M)$ be the set of all second submodules of $M$. For any submodule $N$ of  $M$, $V^{s*}(N)$ is defined to be the set of all second submodules of $M$ contained in $N$. Of course, $V^{s*}(0)$ is just the empty set and $V^{s*}(M)$ is $Spec^s(M)$. It is easy to see that for any family of submodules $N_i$ ($ i \in I$) of $M$, $\cap_{i \in I}V^{s*}(N_i)=V^{s*}(\cap_{i \in I}N_i)$. Thus if $\zeta^{s*} (M)$ denotes the collection of all subsets $V^{s*}(N)$ of $Spec^s(M)$, then $\zeta^{s*}(M)$ contains the empty set and $Spec^s(M)$, and $\zeta^{s*}(M)$ is closed under arbitrary intersections. In general $\zeta^{s*}(M)$ is not closed under finite unions. A module $M$ is called a \emph{cotop module} if $\zeta^{s*}(M)$ is closed under finite unions. In this case, $\zeta^{s*}(M)$ is called
the \emph{quasi Zariski topology} \cite{AF21}.

Now let $N$ be a submodule of $M$. We define $W^s(N)=Spec^s(M)- V^{s*}(N)$ and put $\Omega^s(M)=\{W^s(N):N\leq M \}$. Let $\eta^s(M)$  be the topology on $Spec^s(M)$ by the sub-basis $\Omega^s(M)$. In fact $\eta^s(M)$ is the collection $U$ of all unions of finite intersections of elements of $\Omega^s(M)$ \cite{Mu75}. We call this topology the \emph{second classical Zariski topology} of $M$. It is clear that
if $M$ is a cotop module, then its related topology, as it was mentioned in the above paragraph, coincide with the second
classical Zariski topology. In this paper, we obtain some
new results analogous to those for classical Zariski topology considered in \cite{BH08} and \cite{BHa0}. In Section 2
of this paper, among other results, we investigate the relationship
between the module theoretic properties of $M$ and the topological properties of $Spec^s(M)$ (see Proposition \ref{p2.2}, Corollary \ref{c2.5}, and Theorems \ref{t2.6}, \ref{t2.10}, \ref{t2.12}, and \ref{t2.13}). Moreover,
Theorems \ref{t2.13} and \ref{t3.75} provide some useful characterizations for those modules
whose second classical Zariski topologies are cofinite topologies.

Following M. Hochster \cite{Ho69}, we say that a topological space $W$ is a spectral space if $W$ is homeomorphic to
$Spec(S)$, with the Zariski topology, for some ring $S$.
 Spectral spaces have been characterized by M. Hochster as quasi-compact $T_0$-spaces $W$ having
a quasi-compact open base closed under finite intersection and each irreducible closed subset of $W$
has a generic point \cite{Ho69}. In Section 3, we follow the Hochster's characterization and consider
$Spec^s(M)$ from point view of spectral spaces. We prove that
if $M$ has dcc on socle submodules, then for each $n \in \Bbb N$,
and submodule $N_i$ $(1\leq i \leq n)$ of $M$, $W^s{(N_1)} \cap W^s{(N_2)} \cap, ... \cap W^s{(N_n)}$, in particular $Spec^s(M)$, is quasi compact with second classical Zariski topology (see Theorem \ref{t3.10}).
 It is shown that if $M$ is a finite $R$-module, then $Spec^s(M)$ is a spectral space (see Theorem \ref{t3.31}). Also, it is proved that if $M$ is an $R$-module such that $M$ has dcc on socle submodules, then
$Spec^s(M)$ is a spectral space (see Theorem \ref{t3.11}). Moreover, we show
that if $R$ is a commutative Noetherian ring and $M$ is a comultiplication $R$-module with finite length, then
$Spec^s(M)$ is spectral (see Proposition \ref{p3.30}).

In the rest of this paper, for an $R$-module $M$, $X^s(M)$ will denote $Spec^s(M)$.

\section{Topology on $Spec^s(M)$}
We will consider the cases an $R$-module $M$ when satisfies the following condition:
$$
(**)\\\ For\\\ any\\\ submodules\\\ N_1,N_2 \leq M, \\\ V^{s*}(N_1)=V^{s*}(N_2)\Rightarrow N_1=N_2.
$$
\begin{ex}\label{e2.1}
Every vector space satisfies the $(**)$ condition.
\end{ex}

We recall that for an $R$-module $M$, the \emph{second socle} of $M$ is defined to be the sum of all second submodules of $M$ and denoted by $soc(M)$. If $M$ has no second submodule, then $soc(M)$ is defined to be $0$. Also, a submodule $N$ of $M$ is said to be a \emph{socle submodule} of $M$ if $soc(N)=N$ \cite{AF12}.

\begin{prop}\label{p2.2}
Let $M$ be a nonzero $R$-module. Then the following statements are equivalent.
\begin{itemize}
  \item [(a)] $M$ satisfies the $(**)$ condition.
  \item [(b)] Every nonzero submodule of $M$ is a socle submodule of $M$.
\end{itemize}
\end{prop}
\begin{proof}
$(a)\Rightarrow (b)$. Let $S_1$ be a nonzero submodule of $M$. We claim that $V^{s*}(S_1)
\not =\emptyset$. Otherwise, $V^{s*}(S_1)=\emptyset=V^{s*}(0)$ implies that $S_1=0$, a contradiction. Now let $S_2=\sum_{\acute{S} \in V^{s*}(S_1)}\acute{S}$. Clearly, $V^{s*}(S_1)=V^{s*}(S_2)$. So by hypothesis, $S_1=S_2$. Hence $S_1$ is a socle submodule of $M$.

$(b)\Rightarrow (a)$. This follows from the fact that every submodule $N$ of $M$ is a sum of second submodules if and only if $N=\sum_{S \in V^{s*}(N)}S$.
\end{proof}

\begin{cor}\label{c2.3}
Every semisimple $R$-module $M$ satisfies the $(**)$ condition.
\end{cor}
\begin{proof}
This follows from the fact that every minimal submodule of $M$ is a second submodule of $M$ by \cite[1.6]{Y01}
\end{proof}

Let $M$ be an $R$-module. A proper submodule $P$ of $M$ is said to be
a \emph{semiprime submodule} if $I^2N \subseteq P$, where $N\leq M$ and $I$ is an ideal of $R$, then
$IN\subseteq P$. $M$ is said to be \emph{fully semiprime} if each
proper submodule of $M$ is semiprime.

A nonzero submodule $N$ of $M$ is said to be
\emph{semisecond} if $rN = r^2N$ for each $r\in R$ \cite{AF11}.

\begin{rem}\label{r2.60}
Let $M$ be an $R$-module and $N$ be a submodule of $M$. Let denote the set of all
prime submodules of $M$ by $Spec_R(M)$. Define $V(N)= \{P \in Spec_R(M): P\supseteq N \}$.
An $R$-module $M$ is said to satisfy the $(*)$ condition provided that
if $N_1, N_2 $ are submodules of $M$
with $V(N_1)= V(N_2)$, then $N_1=N_2$ \cite{BH08}.
\end{rem}

\begin{defn}\label{d2.4}
We call an $R$-module $M$ \emph{fully semisecond} if each nonzero submodule of $M$ is semisecond.
\end{defn}

\begin{cor}\label{c2.5}
Let $M$ be an $R$-module. Then the following statements are equivalent.
\begin{itemize}
  \item [(a)] $M$ satisfies the $(*)$ condition.
  \item [(b)] $M$ is a fully semiprime module.
  \item [(c)] $M$ is a cosemisimple module.
  \item [(d)] $M$ satisfies the $(**)$ condition.
  \item [(e)] $M$ is a fully semisecond module.
\end{itemize}
\end{cor}
\begin{proof}
$(a)\Leftrightarrow (b) \Leftrightarrow (c)$. By \cite[2.6]{BH08}.

$(d)\Leftrightarrow (e)$. By Proposition \ref{p2.2}.

$(b)\Leftrightarrow (e)$. By \cite[4.8]{AF11}.
\end{proof}
The \emph{second submodule dimension} of an $R$-module $M$, denoted by $S.dim M$, is defined to be the supremum of the length of chains of second submodules of $M$ if $X^s(M) \not =\emptyset$ and $-1$ otherwise \cite{AF25}.

Let $X$ be a topological space and let $x$ and $y$ be points in $X$. We say that $x$ and $y$
can be \emph{separated} if each lies in an open set which does not contain the other point.
$X$ is a \emph{$T_1$-space} if any two distinct points in $X$ can be separated. A topological
space $X$ is a $T_1$-space if and only if all points of $X$ are closed in $X$.

\begin{thm}\label{t2.6}
Let $M$ be an $R$-module. Then $X^s(M)$ is a $T_1$-space if and only if $S.dim(M)\leq 0$.
\end{thm}
\begin{proof}
First assume that $X^s(M)$ is a $T_1$-space. If $X^s(M)=\emptyset$, then $dim(M)=-1$. Also, if $X^s(M)$ has one element, clearly $S.dimM=0$. So we can assume that $Spec^s(M)$ has more than two elements. We show that every element of $X^s(M)$ is minimal. To show this, let $S_1 \subseteq S_2$, where $S_1, S_2 \in X^s(M)$. Now $\{S_2\}=\cap _{i \in I}(\cup_{j=1}^{n_i}V^{s*}(N_{ij}))$, where $N_{ij} \leq M$ and $I$ is an index set. So for each $i \in I$, $S_2 \in \cup_{j=1}^{n_i}V^{s*}(N_{ij}))$ so that $S_2 \in V^{s*}(N_{it})$, $1\leq t\leq n_i$. Thus $S_1 \in V^{s*}(N_{it})$ for  $1\leq t\leq n_i$. This implies that $S_1 \in \cup_{j=1}^{n_i}V^{s*}(N_{ij}))$ for each $i$. Therefore, $S_1 \in \cap _{i \in I}(\cup_{j=1}^{n_i}V^{s*}(N_{ij}))=\{S_2\}$ as desired.

Conversely, suppose that $S.dim(M)\leq0$. If $S.dim(M)=-1$, then $X^s(M)=\emptyset$ and hence it is a $T_1$-space. Now let $S.dim(M)=0$. Then $X^s(M)\not = \emptyset$ and for every second submodule $S$ of $M$, $V^{s*}(S)=\{S\}$. Hence  $X^s(M)$ is a $T_1$-space.
\end{proof}
\begin{prop}\label{p2.7}
For every finitely cogenerated $R$-module $M$, the following are equivalent.
\begin{itemize}
  \item [(a)] $M$ is a semisimple module with $S.dim(M)=0$.
  \item [(b)] $X^s(M)$ is a $T_1$-space and $M$ satisfies the $(**)$ condition.
\end{itemize}
\end{prop}
\begin{proof}
$(a)\Rightarrow (b)$. By Corollary \ref{c2.3} and Theorem \ref{t2.6}.

$(b)\Rightarrow (a)$. Since $X^s(M)$ is a $T_1$-space, $S.dim(M)\leq 0$ by Theorem \ref{t2.6}. As $M$ is finitely cogenerated,
 every second submodule of $M$ is a minimal submodule of $M$. Now the claim follows from Proposition \ref{p2.2}
\end{proof}

An $R$-module $M$ is said to be a \emph{comultiplication module} if for every submodule $N$ of $M$ there exists an ideal $I$ of $R$ such that $N=(0:_MI)$, equivalently, for each submodule $N$
of $M$, we have $N=(0:_MAnn_R(N))$ \cite{AF07}. Further $M$ is said to be a \emph{weak comultiplication module}
if $M$ does not have any second submodule or for every second submodule $S$
of $M$, $S = (0 :_M I)$ for some $I$ is an ideal of $R$ \cite{AF1}

\begin{thm}\label{t2.10}
Let $M$ be a finite length module over a commutative Noetherian ring $R$ such that $X^s(M)$ is a $T_1$-space. Then $M$ is a comultiplication module.
\end{thm}
\begin{proof}
By \cite [3.6]{AF1}, it is enough to show that $M$ is a weak comultiplication $R$-module. To see this, let $S$ be a second submodule of $M$. Since $S$ is finitely cogenerated and $S.dim (M)=0$ by Theorem \ref{t2.6}, $S$ is a minimal submodule of $M$. Thus $Ann_R(S)=P$ is a maximal ideal of $R$. Since $S\subseteq (0:_MP)$ and $(0:_MP)$ is a second submodule of $M$ by \cite[1.4]{Y01}, $S=(0:_MP)$, as required.
\end{proof}

\begin{thm}\label{t2.11}
Let $M$ be an $R$-module. If either $R$ is an Artinian ring or $M$ is a Noetherian module, then $M$ has a minimal submodule if and only if $M$ has a second submodule. In addition if $M$ has a second submodule, then every second submodule of $M$ is a semisimple submodule of $M$.
\end{thm}
\begin{proof}
First assume that $R$ is an Artinian ring. Then every prime ideal of $R$ is maximal. Let $S$ be a second submodule of $M$. Then $Ann_R(S)$ is a maximal ideal of $R$. Thus $S$ is a semisimple $R/Ann_R(S)$-module. Hence $S$ has a minimal submodule and so $M$ has minimal submodule. Now let $M$ be a Noetherian $R$-module and $S$ be a second submodule of $M$. Since $S$ is finitely generated, one can see that $S$ is a semisimple $R$-module. Therefore, $S$ has a minimal submodule and so $M$ has a minimal submodule, as required.
\end{proof}

\begin{thm}\label{t2.12}
Let $M$ be an $R$-module. If either $R$ is an Artinian ring or $M$ is a Noetherian $R$-module, then $X^s(M)$ is a $T_1$-space if and only if either $X^s(M)=\emptyset$ or $X^s(M)=Min(M)$, where $Min(M)$ denotes the set of all
minimal submodules of $M$.
\end{thm}
\begin{proof}
By Theorem \ref{t2.6}, $X^s(M)$ is a $T_1$-space if and only if $S.dim(M)\leq0$. First suppose that $S.dim(M)\leq0$. If $S.dim(M)=-1$, then $X^s(M)=\emptyset$. If $S.dim(M)=0$, then $X^s(M)\not = \emptyset$. Let $S$ be a second submodule of $M$. Then by Theorem \ref{t2.11}, $S$ has a minimal submodule. Since $S.dim(M)=0$, $S$ is a minimal submodule of $M$. Hence, $X^s(M)\subseteq Min(M)$. The reverse inclusion is clear. Now suppose that either $X^s(M)=\emptyset$ or $X^s(M)=Min(M)$. In the first case, $S.dim(M)=-1$. In the second case, $S.dim(M)=0$ and hence $X^s(M)$ is a $T_1$-space by Theorem \ref{t2.6}
\end{proof}

The cofinite topology is a
topology which can be defined on every set $X$. It has precisely the empty set and
all cofinite subsets of $X$ as open sets. As a consequence, in the cofinite topology,
the only closed subsets are finite sets, or the whole of $X$.
\begin{thm}\label{t2.13}
Let $M$ ba an $R$-module. Then the following are equivalent.
\begin{itemize}
  \item [(a)] $X^s(M)$ is the cofinite topology.
  \item [(b)] $S.dim(M)\leq0$ and for every submodule $N$ of $M$ either $V^{s*}(N)=X^s(M)$ or $V^{s*}(N)$ is finite.
\end{itemize}
\end{thm}
\begin{proof}
$(a) \Rightarrow (b)$. Assume that $X^s(M)$ is the cofinite topology. Since every cofinite topology satisfies the $T_1$ axiom, $S.dim(M)\leq0$ by Theorem \ref{t2.6}. Now assume that there exists a submodule $N$ of $M$ such that $| V^{s*}(N)|=\infty$ and $V^{s*}(N)\not=X^s(M)$. Then $W^s(N)=X^s(M)-V^{s*}(N)$ is open in $X^s(M)$ with infinite complement, a contradiction.

$(b) \Rightarrow (a)$. Suppose that $S.dim(M)\leq 0$ and for every submodule $N$ of $M$, $V^{s*}(N)=X^s(M)$ or $| V^{s*}(N)|< \infty$. Then the complement of every open set in $X^s(M)$ is of the form $\cap_{i \in I}(\cup_{j=1}^nV^{s*}(N_{ij}))$ which is a finite set or $X^s(M)$ obviously.
\end{proof}

\begin{cor}\label{c2.14}
Let $M$ be an $R$-module such that $X^s(M)$ is finite. Then the following statements are equivalent.
\begin{itemize}
  \item [(a)] $X^s(M)$ is a Hausdorff space.
  \item [(b)] $X^s(M)$ is a $T_1$-space.
  \item [(c)] $X^s(M)$ is the cofinite topology.
  \item [(d)] $X^s(M)$ is discrete.
  \item [(e)] $S.dim(M)\leq 0$.
\end{itemize}
\end{cor}

\begin{lem}\label{l2.16}
Let $M$ be a finite length weak comultiplication module. Then $X^s(M)$ is a cofinite topology.
\end{lem}
\begin{proof}
The result follows from Corollary \ref{c2.14} $(e)\Rightarrow (c)$ because $M$ has a finite number of second submodules and every second submodule of $M$ is minimal by \cite[3.4]{AF1}.
\end{proof}

\begin{cor}\label{c2.17}
Let $R$ be a Noetherian ring and $M$ be a finitely generated cocyclic $R$-module. Then $M$ is Artinian and $X^s(M)$ is a cofinite topology.
\end{cor}
\begin{proof}
By \cite{va65}, $M$ is Artinian. Also, $M$ is a comultiplication $R$-module \cite[2.5]{AF13}. Now the result follows from the above Lemma.
\end{proof}

The following example shows that the converse of the above corollary is not true in general.
\begin{ex}\label{e2.18}
Consider $M=\Bbb Z_6$ as a $\Bbb Z$-module. Then $M$ is an Artinian $\Bbb Z$-module and $X^s(M)$ is a cofinite topology but $M$ is not a cocyclic $\Bbb Z$-module.
\end{ex}

\begin{thm}\label{t2.19}
Let $M$ be an $R$-module with $| X^s(M)|\geq 2$. If $X^s(M)$ is a Hausdorff space, then $S.dim(M)=0$ and there exist submodules $N_1, N_2,..., N_n$ of $M$ such that $V^{s*}(N_i)\not =X^s(M)$, for all $i$, and $X^s(M)=\cup_{i=1}^nV^{s*}(N_i)$.
\end{thm}
\begin{proof}
The proof is similar to that of \cite[2.26]{BH08}.
\end{proof}

Maximal second submodules are defined in a natural way. By Zorn's Lemma one can easily see that each second submodule of a module $M$ is contained in a maximal second submodule of $M$ \cite{AF12}. In \cite{AF12} and \cite{AF1}, it is shown that Artinian modules and Noetherian modules contain only finitely many maximal second submodules.
\begin{cor}\label{c2.20}
Let $M$ be an Artinian $R$-module. Then the following statements are equivalent.
\begin{itemize}
  \item [(a)] $X^s(M)$ is a Hausdorff space.
  \item [(b)] $X^s(M)$ is a $T_1$-space.
  \item [(c)] $X^s(M)$ is the cofinite topology.
  \item [(d)] $X^s(M)$ is discrete.
  \item [(e)] $X^s(M)=Min(M)$.
  \end{itemize}
\end{cor}
\begin{proof}
First we note that since $M$ is Artinian, $X^s(M)$ is not empty.

$(a)\Rightarrow (b)$. This is clear.

$(b)\Rightarrow (c)$. By Theorem \ref{t2.6}, $S.dim(M)\leq 0$. Thus every second submodule of $M$ is a maximal second submodule of $M$. As $M$ is Artinian, it contains only finitely many maximal second submodules by \cite[2.6]{AF12}. Therefore, $X^s(M)$ is finite. Hence, $X^s(M)$ is a cofinite topology by Corollary \ref{c2.14}.

$(c)\Rightarrow (d)$. By Theorem \ref{t2.13}, $S.dim(M)\leq0$ so that, as we saw in the proof of $(b)\Rightarrow (c)$, $X^s(M)$ is finite. Now the result follows from the corollary \ref{c2.14}.

$(d)\Rightarrow (e)$. Since $X^s(M)$ is a $T_1$-space, by Theorem \ref{t2.6}, $S.dim(M)=0$. Since $M$ is Artinian, every second submodule of $M$ contains a minimal submodule of $M$. Therefore, every second submodule of $M$ is minimal so that $X^s(M)\subseteq Min(M)$. The reverse inclusion follows from the fact that every minimal submodule of $M$ is second by \cite[1.6]{Y01}.

$(e)\Rightarrow (a)$. Since $M$ is Artinian, $M$ contains only a finite number of maximal second submodule by \cite[2.6]{AF12}. As $X^s(M)=Min(M)$, every maximal second submodule of $M$ is minimal. Therefore, $S.dim(M)\leq 0$ and $X^s(M)$ is finite. Now the result follows from Corollary \ref{c2.14}.
\end{proof}

\begin{thm}\label{t3.75}
Let $M$ be a Noetherian $R$-module. Then the following statements are equivalent.
\begin{itemize}
  \item [(a)] $X^s(M)$ is a Hausdorff space.
  \item [(b)] $X^s(M)$ is a $T_1$-space.
  \item [(c)] $X^s(M)$ is the cofinite topology.
  \item [(d)] $X^s(M)$ is discrete.
  \item [(e)] Either $X^s(M)=\emptyset$ or $X^s(M)=Min(M)$.
\end{itemize}
\end{thm}
\begin{proof}
$(a)\Rightarrow (b)$. This is clear.

$(b)\Rightarrow (c)$. Let $X^s(M)$ be a $T_1$-space. By Theorem \ref{t2.12}, either $X^s(M)=\emptyset$ or $X^s(M)=Min(M)$. Let $X^s(M) \not =\emptyset$. Then by using \cite[2.2]{AF12}, every second submodule of $M$ is a maximal second submodule of $M$. Since $M$ is Noetherian, it has a finite number of maximal second submodules by \cite[2.4]{AF1}. Thus, $X^s(M)$ is finite and so by Corollary \ref{c2.14}, $X^s(M)$ is the cofinite topology.

$(c)\Rightarrow (d)$. Assume that $X^s(M)$ is the cofinite topology. Then by Theorem \ref{t2.13}, $S.dim(M)\leq 0$. Now as we see in the proof of $(b)\Rightarrow (c)$, $X^s(M)$ is finite. Therefore, $X^s(M)$ is discrete by Corollary \ref{c2.14}.

$(d)\Rightarrow (e)$. This follows from Theorem \ref{t2.12}

$(e)\Rightarrow (a)$. If $X^s(M)=Min(M)$, then by using \cite[2.2]{AF12}, every second submodule of $M$ is a maximal second submodule of $M$. As $M$ is Noetherian, $M$ has a finite number of maximal second submodules. Therefore, $X^s(M)$ is finite. Now by Corollary \ref{c2.14}, $X^s(M)$ is a Hausdorff space.
\end{proof}

\begin{lem}\label{c2.22}
Let $M$ be a second module. Then $X^s(M)$ is $T_1$-space if and only if $M$ is the only second submodule of $M$.
\end{lem}
\begin{proof}
This follows from Theorem \ref{t2.6}.
\end{proof}
\begin{lem}\label{l3.1}
Let $f: \acute{M} \rightarrow M$ be an $R$-module monomorphism, and let $N$ be a
submodule of $M$ such that $N \subseteq f(\acute{M})$. Then $V^{s*}(N)\rightarrow V^{s*}(f^{-1}(N))$, given by $S\rightarrow f^{-1}(S)$ is a bijection. If $V^{s*}(N)= \emptyset$, then so is $V^{s*}(f^{-1}(N))$.
\end{lem}
\begin{proof}
It is straightforward.
\end{proof}

\begin{thm}\label{t3.2}
Let $f:\acute{M} \rightarrow M$ be an $R$-module homomorphism. Define $\varphi: \Omega^s({M})\rightarrow
\Omega^s(\acute M)$ by $\varphi(\cap_{i \in I}(\cup_{j=1}^{n_i}V^{s*}(N_{ij})))=\cap_{i \in I}(\cup_{j=1}^{n_i}V^{s*}(f^{-1}(N_{ij})))$, where $I$ is an index set, $n_i \in \Bbb N$ and $N_{ij}\leq M$. Then $\varphi$ is a well-defined map.
\end{thm}
\begin{proof}
Suppose that $\cap_{i \in I}(\cup_{j=1}^{n_i} V^{s*}(N_{ij}))=\cap_{t \in T}(\cup_{j=1}^{n_t}V^{s*}(K_{tj}))$, where $N_{ij} , K_{tj} \leq M$ and $I, T$ are index sets. We show that
$$
\cap_{i \in I}(\cup_{j=1}^{n_i} V^{s*}(f^{-1}(N_{ij})))=\cap_{t \in T}(\cup_{j=1}^{n_t}V^{s*}(f^{-1}(K_{tj})))    \\\ (1)
$$
Let $S \in \cap_{i \in I}(\cup_{j=1}^{n_i} V^{s*}(f^{-1}(N_{ij})))$. Then for each $i \in I$, there exists $j_i$ ($1\leq j_i\leq n_i$) such that $S \in V^{s*}(f^{-1}(N_{ij_i}))$. If $S\subseteq Ker(f)$, then for each $t \in T$, and each $j$ ($1\leq j\leq n_t$) we have $S \subseteq f^{-1}(K_{tj})$. It follows that $S \in \cap_{t \in T}(\cup_{j=1}^{n_t}V^{s*}(f^{-1}(K_{tj})))$.
Now let $S \not \subseteq Ker(f)$. Then $f(S)$ is a second submodule of $M$. Hence for each $i \in I$, $f(S) \in V^{s*}(N_{ij})$. Thus $f(S) \in \cap _{i \in I}(\cup_{j=1}^{n_i}V^{s*}(N_{ij}))$. Therefore, $f(S) \in \cap_{t \in T}(\cup_{j=1}^{n_t}V^{s*}(K_{tj}))$, and hence for each $t \in T$, there exists $j_t$ ($1\leq j_t\leq n_t$) such that $f(S) \in V^{s*}(K_{tj_t})$. It follows that for each $t \in T$, $S\subseteq f^{-1}(K_{tj_t})$ so that $S \in V^{s*}(f^{-1}(K_{tj_t}))$. Consequently,
we have $S \in \cap_{i \in T}(\cup_{j=1}^{n_t}V^{s*}(f^{-1}(K_{tj})))$. Therefore,
$$
\cap_{i \in I}(\cup_{j=1}^{n_i} V^{s*}(f^{-1}(N_{ij})))\subseteq \cap_{i \in T}(\cup_{j=1}^{n_t}V^{s*}(f^{-1}(K_{tj}))).
$$
 By a similar argument we see that
$$
 \cap_{i \in T}(\cup_{j=1}^{n_t}V^{s*}(f^{-1}(K_{tj})))\subseteq \cap_{i \in I}(\cup_{j=1}^{n_i} V^{s*}(f^{-1}(N_{ij}))).
$$
Thus $(1)$ holds.
\end{proof}

\begin{prop}\label{p3.3}
Let $f:\acute{M}\rightarrow M$ be a monomorphism such that $\acute{M}$ is not secondless.
Define $\nu: X^s(\acute{M})\rightarrow X^s(M)$ by $\nu (S)=f(S) \in X^s(M)$ for each $S \in X^s(\acute{M})$. Then $\nu$ is a continuous map.
\end{prop}
\begin{proof}
Clearly, $\nu$ is well-defined. Let $V=\cap_{i \in I}(\cup_{j=1}^{n_i}V^{*s}(N_{ij}))$ be a closed set in $X^s(M)$. We show that $\nu^{-1}(V)=\cap_{i \in I}(\cup_{j=1}^{n_i}V^{*s}(f^{-1}(N_{ij}))$. Let $S \in \nu^{-1}(V)$. Then $\nu(S) \in V$, so $f(S) \in \cap_{i \in I}(\cup_{j=1}^{n_i}V^{*s}(N_{ij}))$. Therefore, for each $i \in I$, there exists $j_i$ such that $f(S) \in V^{*s}(N_{ij_i})$.
But $\phi: V^{s*}(N_{ij_i}\cap Im(f))\rightarrow V^{s*}(f^{-1}(N_{ij_i}))$, given by $\acute{S}\rightarrow f^{-1}(\acute{S})$, is a bijective map by Lemma \ref{l3.1}. Hence we have $\phi (f(S))=f^{-1}f(S)=S+Ker(f)=S\in V^{s*}(f^{-1}(N_{ij_i}))$ by Lemma \ref{l3.1}.
It follows that $\nu^{-1}(V)\subseteq \cap_{i \in I}(\cup_{j=1}^{n_i}V^{s*}(f^{-1}(N_{ij}))$. The reverse inclusion is proved similarly and the proof is completed.
\end{proof}

\begin{rem}\label{r100}
Clearly, for an $R$-module $M$, $X^s(M)=X^s(soc(M))$. This fact shows that the study of Zariski topology on the second spectrum of $M$ can be easily reduced to that of socle modules.
\end{rem}

\begin{lem}\label{l3.6}
Let $M$ be an $R$-module and $S \in X^s(M)$. Let $V^{s*}(S)$ be endowed with the induced topology of $X^s(M)$. Then $V^{s*}(S)=X^s(S)$, where $S$ is a second submodule of $M$.
\end{lem}
\begin{proof}
Straightforward
\end{proof}

\begin{prop}\label{p3.7}
Let $M$ be an $R$-module, $I_M:=Ann_R(Soc_R(M))$ and $\overline{R}=R/I_M$. Then
the natural map $\psi_M:X^s(M)\rightarrow Spec(R/I_M)$ given by $S\rightarrow \overline{Ann_R(S)}=Ann_R(S)/I_M$ is a continuous map. Moreover, $\psi_M=\psi_{soc(M)}$.
\end{prop}
\begin{proof}
Let $A$ be a closed subset of $Spec(\overline{R})$. Then $A=V(\overline{I})$ for some ideal $\overline{I}$ of $\overline{R}$. We claim that $\psi_M^{-1}(V(\overline{I}))=V^{s*}((0:_MI))$. So let $S \in \psi_M^{-1}(V(\overline{I}))$. Then $\psi(S) \in V(\overline{I})$. Hence $\overline{I}\subseteq \overline{Ann_R(S))}$ and so $I \subseteq Ann_R(S)$. Thus $S \subseteq (0:_MI)$ so that $S \in V^{s*}(0:_MI)$.
To see the reverse inclusion, let $S \in V^{s*}((0:_MI))$. Then $S \subseteq (0:_MI)$. Hence $I \subseteq Ann_R(S)$. Therefore, $\overline{I}\subseteq \overline{Ann_R(S)}$. Hence $\overline{Ann_R(S)} \in V(\overline{I)}$. This implies that $\psi(S) \in V(\overline{I})$ and so $S \in \psi^{-1}(V(\overline{I}))$. Therefore, $\psi_M^{-1}(V(\overline{I}))=V^{s*}((0:_MI))$ and hence $\psi_M$ is continuous. The second assertion follows from the Remark \ref{r100} .
\end{proof}

\section{Modules whose second classical Zariski topologies are spectral spaces}
Let $Z$ be a subset of a topological space $W$. Then the notation $cl(Z)$ will
denote the closure of $Z$ in $W$.

A topological space $X$ is called irreducible if $X \not =\emptyset$ and every finite intersection of non-empty open subset of $X$ is non-empty. A non-empty subset $Y$ of a topology $X$ is called \emph{irreducible set} if the subspace $Y$ of $X$ is irreducible, equivalently if $Y_1$ and $Y_2$ are closed subset of $X$ and satisfy $Y\subseteq Y_1 \cup Y_2$, then $Y\subseteq Y_1$ or $Y\subseteq Y_2$.

Let $Y$ be a closed subset of a topological space. An element $y \in Y$ is called a \emph{generic point} of $Y$ if $Y=cl(\{y \})$. A generic point of an
irreducible closed subset of a topological space is unique if the
topological space is a $T_{0}$ space.

A spectral space is a topological space homeomorphic to the prime spectrum of a commutative ring equipped
with the Zariski topology. Spectral spaces have been characterized by Hochster \cite [p.52, Prop.4] {Ho69} as the topological spaces $W$ which satisfy the following conditions:
\begin{itemize}
  \item [(a)] $W$ is a $T_0$-space;
  \item [(b)] $W$ is quasi-compact;
  \item [(c)] the quasi-compact open subsets of $W$ are closed under finite intersection and form an open base;
  \item [(d)] each irreducible closed subset of $W$ has a generic point.
\end{itemize}

Let $M$ be an $R$-module and $N$ be a submodule of $M$.
In \cite{BH08}, among other nice results, Proposition 3.1 states
that if $Y$ ia a nonempty subset of $Spec_R(M)$, then $cl(Y)= \cup_{P \in Y} V(P)$, where
$V(N)= \{P \in Spec_R(M)\mid P \supseteq N \}$.  Unfortunately, this result is not true in general because if we take $M=\Bbb Z$, where $\Bbb Z$ is the ring of integers, and $Y= Max(\Bbb Z)$, then we have $cl(Y)= Max(\Bbb Z)$, while
$cl(Y)= V(\cap_{P\in Y} P)= V(0)= \Bbb Z$ by \cite[5.1]{lu99}. However, this theorem is true when $Y$ is a finite set which
has been used by the authors during their results in \cite{BH08} and \cite{BHa0}.

Let $M$ be an $R$-module and $Y$ a subset of $X^s(M)$. We will denote $ \sum_{S \in Y}S$ by $T(Y)$.

\begin{prop}\label{p2.200}
Let $M$ be an $R$-module.
\begin{itemize}
\item [(a)] If $Y$ is a finite subset of $X^s(M)$, then
$cl(Y )= \cup_{S \in Y} V^{s*}(S)$.
\item [(b)]
if $Y$ is a closed subset of $X^s(M)$, we have  $Y= \cup_{S \in Y} V^{s*}(S)$.
\item [(c)]
If $M$ is a cotop $R$-module and if $Y$ is a subset of $X^s(M)$,
then $cl(Y )=V^{s*}(T(Y))= V^{s*}(\sum_{S \in Y} S).$

\end{itemize}
\end{prop}
\begin{proof}
(a) Clearly, $Y \subseteq \cup_{S \in Y} V^{s*}(S)$.  Let $K$ be a closed subset of $X^s(M)$ containing $Y$.
 Thus $K= \cap_{i \in I}(\cup_{j=1}^{n_i}V^{s*}(N_{ij}))$, for some $N_{ij} \subseteq M$,
 $i \in I$, and $n_i \in \Bbb N$. Let $S_1 \in \cup_{S \in Y} V^{s*}(S)$. Then there exists $S \in Y$ such that $S_1 \in V^{s*}(S)$ so that $S_1 \subseteq S$. But $S \in Y$ implies that for each $i \in I$, there exists $j$, $1 \leq j \leq n_i$, such that $S \subseteq N_{ij}$. Thus we have $S_1\subseteq S \subseteq  N_{ij}$. Therefore $S_1 \in K$ so that
$\cup_{S \in Y} V^{s*}(S) \subseteq K$. In other words, $\cup_{S \in Y} V^{s*}(S)$ is the smallest closed subset of
$X^s(M)$ containing $Y$, i.e., $cl(Y)= \cup_{S \in Y} V^{s*}(S)$.

(b) It is enough to show  $\cup_{S \in Y}V^{s*}(S) \subseteq Y$ because the other side is clear.
To see this, we note that for every element $S$ of $Y$,
 $V^{s*}(S)= cl(\{S \}) \subseteq cl(Y)=Y$ by part (a). Hence $\cup_{S \in Y}V^{s*}(S) \subseteq Y$, as required.

(c) First we note that since $M$ ia a cotop module, each closed set is of the form of $ V^{s*}(N)$ for some $N \leq M$. Clearly $Y \subseteq V^{s*}(T(Y))$. Now let $ V^{s*}(N)$ be any closed subset of $X^s(M)$ containing $Y$.
Then for each $S \in Y$, we have $ S \subseteq N$ so that $T(Y) \subseteq N$. So if $S \in V^{s*}(T(Y))$,
then $S \subseteq T(Y) \subseteq N$. Hence $S \in V^{s*}(N)$, i.e., $V^{s*}(T(Y)) \subseteq V^{s*}(N)$.
This shows that $cl(Y )=V^{s*}(T(Y))$.  This completes the proof.
\end{proof}

\begin{cor}\label{c2.24}
Let $M$ be an $R$-module. Then we have the following.
\begin{itemize}
  \item [(a)] $cl(\{S \})=V^{s*}(S)$, for all $S \in Spec^s(M)$.
  \item [(b)] $S_1 \in cl(\{S \}) \Leftrightarrow S_1 \subseteq S \Leftrightarrow V^{s*}(S_1)\subseteq V^{s*}(S)$.
  \item [(c)] The set $\{S \}$ is closed in $X^s(M)$ if and only if $S$ is a minimal second submodule of $M$.
\end{itemize}
\begin{proof}
Use Proposition \ref{p2.200}(a).
\end{proof}
\end{cor}

\begin{lem}\label{c22.24}
Let $M$ be an $R$-module. Then for each $S \in X^s(M)$, $V^{s*}(S)$ is irreducible. In particular, $X^s(M)$ is irreducible.
\end{lem}
\begin{proof}
The proof is straightforward.
\end{proof}

We need the following evident Lemma.

\begin{lem}\label{l2.25}
Let $S$ be a submodule of an $R$-module $M$. Then the following are equivalent.
\begin{itemize}
  \item [(a)] $S$ is a second submodule of $M$.
  \item [(b)] For each $r \in R$ and submodule $K$ of $M$, $rS\subseteq K$ implies that $rS=0$ or $S\subseteq K$.
\end{itemize}
\end{lem}
\begin{proof}
The proof is straightforward.
\end{proof}

\begin{thm}\label{t2.25}
Let $M$ be an $R$-module and $Y \subseteq X^s(M)$.
\begin{itemize}
  \item [(a)] If $Y$ is irreducible, then $T(Y)$ is a second submodule.
  \item [(b)] If $T(Y)$ is a second submodule and $T(Y) \in cl({Y})$, then $Y$ is irreducible.
\end{itemize}
\end{thm}
\begin{proof}
 (a) Assume that $Y$ is an irreducible subset of $X^s(M)$. Then obviously $T(Y)$ is a nonzero submodule of $M$ and $Y\subseteq V^{s*}(T(Y))$. Now let $rT(Y)\subseteq K$, whence $r \in R$ and $K$ is a submodule of $M$. It is easy to see that $Y\subseteq V^{s*}((K:_Mr))\subseteq V^{s*}(K) \cup V^{s*}((0:_Mr))$. Since $Y$ is irreducible, we have $Y\subseteq V^{s*}((0:_Mr))$ or $Y\subseteq V^{s*}(K)$. If $Y\subseteq V^{s*}((0:_Mr))$, then $rS=0$ for all $S \in Y$. Thus $rT(Y)=0$. If $Y\subseteq V^{s*}(K)$, then $S\subseteq K$ for each $S \in Y$, so $T(Y)\subseteq K$. Hence by Lemma \ref{l2.25}, $T(Y)$ is a second submodule of $M$.

(b) Assume that $S:= T(Y)$ is a second submodule of $M$ and $S \in cl({Y})$. It is easy to see that $cl({Y})=V^{s*}(S)$. Now let $Y\subseteq Y_1\cup Y_2$, where $Y_1$, $Y_2$ are closed sets. Then we have $V^{s*}(S)=cl({Y})\subseteq Y_1 \cup Y_2$. Since $V^{s*}(S)$ is irreducible, $V^{s*}(S)\subseteq Y_1$ or $V^{s*}(S)\subseteq Y_2$. Hence $Y\subseteq Y_1$ or $Y\subseteq Y_2$. So $Y$ is irreducible.
\end{proof}

\begin{cor}\label{c2.27}
Let $M$ be an $R$-module and let $N$ be a submodule of $M$. Then $V^{s*}(N)$ of $Spec^s(M)$ is irreducible if and only if $soc(N)$ is a second submodule of $M$. Consequently, $Spec^s(M)$ is irreducible if and only if $soc(M)$ is a second module.
\end{cor}
\begin{proof}
($\Rightarrow$). Let $Y: = V^{s*}(N)$ be an irreducible subset of $X^s(M)$. Then we have $T(Y)= soc(N)$ so
that $soc(N)$ is a second submodule of $M$ by Theorem \ref{t2.25}(a).

($\Leftarrow$).
Suppose $soc(N)$ is a second submodule of $M$. Then by Theorem \ref{p2.200}(b),
$Y: = V^{s*}(N)= \cup_{S \in Y} V^{s*}(S)$ so that $soc(N) \in cl(Y)$. Hence $V^{s*}(N)$ is irreducible by
Theorem \ref{t2.25}(b).
\end{proof}

We remark that any closed subset of a spectral space is spectral for the induced topology, and we note that a generic point of an irreducible closed subset $Y$ of a topological space is unique if the topological space is a $T_0$-space. The following proposition shows that for any $R$-module $M$, $X^s(M)$ is always a $T_0$-space.

\begin{lem}\label{l2.28}
Let $M$ be an $R$-module. Then the following hold.
\begin{itemize}
  \item [(a)] $X^s(M)$ is a $T_0$-space.
  \item [(b)] Let $S \in X^s(M)$. Then $S$ is a generic point of the irreducible closed subset
  $V^{s*}(S)$.
\end{itemize}
\end{lem}
\begin{proof}
(a) This follows from Corollary \ref{c2.24} and the fact that a topological space is a $T_0$-space if and only if the closures of distinct points are distinct.

(b) By Corollary \ref{c2.24}.

\end{proof}

By \cite[3.3]{AF21}, if $M$ is a comultiplication $R$-module $M$, then the second classical Zariski topology of $M$ and the Zariski topology of $M$ considered in \cite{AF21} coincide (note that every comultiplication module is a cotop module).

\begin{prop}\label{p2.29}
Let $R$ be a commutative Noetherian ring and let $M$ be a cotop $R$-module with finite length. Assume that
 the second classical Zariski topology of $M$ and the Zariski topology of $M$ considered in \cite{AF21} coincide. Then $M$ is a comultiplication $R$-module.
\end{prop}
\begin{proof}
Assume that the second classical Zariski topology of $M$ and the Zariski topology of $M$ considered in \cite{AF21}, coincide. Then by Lemma \ref{l2.28}, $Spec^s(M)$ with the topology considered in \cite{AF21} is a $T_0$-space. Now by \cite[7.6]{AF21}, $M$ is a comultiplication $R$-module.
\end{proof}

\begin{prop}\label{p3.30}
Let $R$ be a commutative Noetherian ring and let $M$ be a comultiplication $R$-module with finite length. Then $X^s(M)$ is a spectral space (with the second classical Zariski topology).
\end{prop}
\begin{proof}
This follows from Lemma \ref{l2.28} and \cite[7.5]{AF21} and the fact that since $M$ is a cotop module, its assigned
topology coincide with the the second classical topology. In this case, every closed set can be written as
$V^{s*}(N)$ for some submodule $N$ of $M$.
\end{proof}

\begin{thm}\label{t3.31}
Let $M$ be an $R$-module with finite second spectrum. Then $X^s(M)$ is a spectral space (with the second Zariski topology). Consequently, for each finite module $M$, $Spec^s(M)$ is a spectral space.
\end{thm}
\begin{proof}
Since $X^s(M)$ is a nonempty finite set, every subset of $X^s(M)$ is quasi-compact. So the quasi-compact open sets
of $X^s(M)$ are closed under finite intersection. Further $\beta = \{W^s(N_1) \cap W^s(N_2) \cap, ... \cap W^s(N_n), N_i \leq M , n \in \Bbb N \}$
is a basis for  $X^s(M)$ with the property that each basis element, in particular $W^s(M)= X^s(M)$, is quasi-compact.
Also if $Y= \{y_1, y_2, ..., y_n \}$ is an irreducible subset of $X^s(M)$, then we have $cl(Y)= \cup_{i= 1}^n cl(\{ y_i \}$. Since $Y$ is irreducible, $cl(Y)= cl(\{ y_i \})$ for some $1 \leq i \leq n$.
Moreover, $X^s(M)$ is
a $T_0$-space by Lemma \ref{l2.28}. Therefore $X^s(M)$ is a spectral space by Hochster's characterization.
\end{proof}

\begin{defn}\label{d.500}
Let $M$ be an $R$-module and let $H(M)$ be the family of all subsets of $X^s(M)$ of the form $V^{s*}(N) \cap W(K)$, where $N, K \leq M$. $H(M)$ contains both $X^s(M)$ and $\emptyset$ because
$X^s(M)= V^{s*}(M) \cap W^s(0)$ and $ \emptyset= V^{s*}(M) \cap W^s(M)$. Let $Z(M)$ be the collection of all unions of finite intersections of elements of $H(M)$. Then $Z(M)$ is a topology on $X^s(M)$ and called the
\emph{finer patch topology} or \emph{constructible topology}. In fact, $H(M)$ is a sub-basis for the finer patch topology of $M$.
\end{defn}

\begin{thm}\label{t3.9}
Let $M$ be an $R$-module such that $M$ satisfies descending chain condition of socle submodules. Then $X^s(M)$ with the finer patch topology is a compact space.
\end{thm}
\begin{proof}
Let $A$ be a family of finer patch-open sets covering $X^s(M)$, and suppose that no finite subfamily of $A$ covers $X^s(M)$. Since $V^{s*}(soc(M))=V^{s*}(M)=X^s(M)$, we may use the $dcc$ on socle submodules to choose a submodule $N$ minimal with respect to the property that no finite subfamily of $A$ covers $V^{s*}(N)$ (note that we may assume $N=soc(N)$ because $V^{s*}(N)=V^{s*}(soc(N))$. We claim that $N$ is a second submodule of $M$ , for if not, then there exist a submodule $L$ of $M$ and $r \in R$ such that $rN\subseteq L$, $rN \not =0$, and $N \not \subseteq L$. Thus $N \cap L\subset N$ and $Soc(N \cap (0:_Mr))\subseteq N \cap (0:_Mr)\subset N$. Hence without loss of generality, there must be a finite subfamily $\acute{A}$ of $A$ that covers both $V^{s*}(N \cap (0:_Mr))$ and $V^{s*}(N \cap L)$. Let $S \in V^{s*}(N)$. Since $rN \subseteq L$, we have $rS\subseteq L$. Since $S$ is second, $S\subseteq L$ or $rS=0$. Thus either $S \in V^{s*}(N \cap L)$ or $S \in V^{s*}(N \cap (0:_Mr))$. Therefore,
$$
V^{s*}(N)\subseteq V^{s*}(N \cap (0:_Mr)) \cup V^{s*}(N \cap L).
$$
Thus $V^{s*}(N)$ is covered with the finite subfamily $\acute{A}$, a contradiction. Hence $N$ is a second submodule of $M$. Now choose $U \in A$ such that $N \in U$. Thus $N$ must have a patch-neighborhood $\cap_{i=1}^n(W^s(K_i) \cap V^{s*}(N_i))$, for some $K_i, N_i\leq M$, $n \in \Bbb N$, such that
$$
\cap_{i=1}^n[W^s(K_i) \cap V^{s*}(N_i)]\subseteq U.
$$
We claim that for each $i$ ($1\leq i\leq n$),
 $$
 N \in W^s(K_i \cap N) \cap V^{s*}(N)\subseteq W^s(K_i) \cap V^{s*}(N_i).
 $$
 To see this, assume that $S \in W^s(K_i \cap N) \cap V^{s*}(N)$ so that $S \not \subseteq K_i \cap N$ and $S \subseteq N$. Thus $S \not \subseteq K_i$, i.e., $S \in W^s(K_i)$. On the other hand, $N \in V^{s*}(N_i)$ and $S \subseteq N$. Thus $S \in V^{s*}(N_i)$. Hence we have
 $$
  N \in \cap_{i=1}^n[W^s(K_i \cap N) \cap V^{s*}(N)]\subseteq \cap_{i=1}^n[W^s(K_i) \cap V^{s*}(N_i)]\subseteq U.
  $$
 Thus $[\cap_{i=1}^nW^s(\acute{K_i})] \cap V^{s*}(N)$, where $\acute{K_i}:= K_i \cap N \subset N$ is a neighborhood of $N$ with $[\cap_{i=1}^n[W^s(\acute{K_i})] \cap V^{s*}(N)\subseteq U$. Since for each $i$ $(1\leq i\leq n)$, $\acute{K_i}\subset N$, $V^{s*}(\acute{K_i})$ can be covered by some finite subfamily $\acute{A_i}$ of $A$. But
  $$
  N^{s*}(N)-[\cup_{i=1}^nV^{s*}(\acute{K_i})]=V^{s*}(N)-[\cap _{i=1}^nW^s(\acute{K_i})]^c=[\cap_{i=1}^n[W^s(\acute{K_i})]\cap V^{s*}(N)\subseteq U.
  $$
  Hence $V^{s*}(N)$ can be covered by $\acute{A_1} \cup \acute{A_2}\cup...\cup \acute{A_n}\cup \{U\}$, contrary to our choice of $N$. Thus there exists a finite subfamily of $A$ which covers $X^s(M)$. Therefore, $X^s(M)$ is compact in the finer patch topology of $M$.
\end{proof}
\begin{prop}\label{p3.800}
Let $M$ be $R$-module such that $M$ has dcc on socle submodules. Then every irreducible
closed subset of $X^s(M)$ (with second classical Zariski topology) has a generic point.
\end{prop}
\begin{proof}
Let $Y$ be an irreducible closed subset of $X^s(M)$. First we note that if $N$ is a submodule of $M$, then $V^{s*}(N)$ and $W^s(N)$ are both open and closed
in finer patch topology. Hence $V^{s*}(S)$, where $S \in Y$, and $Y$ are also an open
and a closed set in finer patch topology respectively. Since $X^s(M)$ is a compact space in finer patch topology by
Theorem \ref{t3.9} and $Y$ is closed in $X^s(M)$, we have $Y$ is a compact space in finer patch topology. Now
$Y= \cup_{S \in Y}V^{s*}(S)$ by Proposition \ref{p2.200}(b) and each $V^{s*}(N)$ is open in finer patch topology. Hence
there exists a finite set $Y_1 \subseteq Y$ such that $Y= \cup_{S \in Y_1}V^{s*}(S)$. Since $Y$ is irreducible,
$Y= V^{s*}(S)= cl(\{S \})$ for some $S \in Y$. Hence $S$ is a generic point for $Y$, as desired.
\end{proof}

\begin{lem}\label{l.300}
Assume $\zeta_1$ and  $\zeta_2$ are two topologies on $X^s(M)$ such that  $\zeta_1 \leq \zeta_2$.
If  $X^s(M)$ is quasi-compact (i.e., every open cover of it has a finite subcover) in $\zeta_2$, then $X^s(M)$ is also
quasi-compact in $\zeta_1$.
\end{lem}

\begin{thm}\label{t3.10}
Let $M$ be an $R$-module such that $M$ has dcc on socle submodules. Then for each $n \in \Bbb N$,
and submodules $N_i$ $(1\leq i \leq n)$ of $M$, $W^s{(N_1)} \cap W^s{(N_2)} \cap, ... \cap W^s{(N_n)}$, is a quasi compact subset of $X^s(M)$ with the second classical Zariski topology.
\end{thm}
\begin{proof}
Clearly, for each $n \in \Bbb N$ and each submodule $N_i$ ($1\leq i \leq n$) of $M$,  $W^s(N_1)
\cap W^s(N_2) \cap, ... \cap W^s(N_n)$ is a closed set in $X^s(M)$ with finer patch topology so that it compact
in $X^s(M)$ with finer patch topology. Hence it is quasi-compact in $X^s(M)$ with the second classical Zariski topology by Lemma \ref{l.300}, as desired.
\end{proof}

\begin{cor}\label {c.300}
Let $M$ be an $R$-module such that $M$ has dcc on socle submodules.
Then quasi-compact open sets of $X^s(M)$ are closed under finite intersections.
\end{cor}
\begin{proof}
It suffices to show that the intersection $U= U_1 \cap U_2$ of two quasi-compact open sets $U_1$ and  $U_2$
of $X^s(M)$ is a quasi-compact set.
Each $U_i$, $i=1, 2$, is a finite union of members of the open base $\beta = \{W^s(N_1) \cap W^s(N_2) \cap, ... \cap W^s(N_n), N_i \leq M, n \in \Bbb N \}$, hence so is $U= \cup_{i=1}^m (\cap_{j=1}^{n_i} W^s(N_j)$.  Let $\Omega$ be any open cover of $U$.
Then $\Omega$ also covers each $\cap_{j=1}^{n_i} W^s(N_j)$ which is quasi-compact by Theorem \ref{t3.10}. Hence each $\cap_{j=1}^{n_i} W^s(N_j)$ has a finite subcover of $\Omega$ and so does $U$.
\end{proof}

\begin{thm}\label{t3.11}
Let $M$ be an $R$-module $R$-module such that $M$ has dcc on socle submodules. Then
$X^s(M)$ with the second classical Zariski topology is a spectral space.
\end{thm}
\begin{proof}
We have $X^s(M)$ is a $T_0$ space by Lemma \ref{l2.28}.
Further $\beta = \{W^s(N_1) \cap W^s(N_2) \cap, ... \cap W^s(N_n), N_i \leq M, n \in \Bbb N \}$
is a basis for $X^s(M)$ with the property that each basis element, in particular $W^s(M)= X^s(M)$, is quasi-compact by Theorem \ref{t3.10}.
 Moreover, by Corollary \ref{c.300}, the quasi-compact open sets are closed under any finite intersections. Finally, every irreducible closed set has a generic point by Proposition \ref{p3.800}.
 Therefore, $X^s(M)$ is a spectral space by Hochster's characterization.
\end{proof}

\textbf{Acknowledgments.} We would like to thank the Dr. R. Ovlyaee-Sarmazdeh for his helpful comments.
\bibliographystyle{amsplain}

\end{document}